\documentclass[reqno,11pt]{amsart}   	

\usepackage{amssymb}
\usepackage[colorlinks=false,backref=true,pagebackref=false,pdftex,pdfauthor={Wendel, Ederson},pdftitle={Monotonicity of the Morse index of radial solutions of the H\'enon equation in dimension two}]{hyperref}

\usepackage{color}
\usepackage{amsmath,amsfonts}
\usepackage{amssymb}
\usepackage{amscd}
\usepackage{enumerate}
\usepackage[utf8]{inputenc}
\usepackage[brazil, english]{babel} 
\usepackage{amsthm}
\usepackage{indentfirst}
\usepackage{latexsym,hyperref}
\usepackage[showonlyrefs]{mathtools}
\mathtoolsset{showonlyrefs=true}
\usepackage{url}

\newtheorem{theorem}{{Theorem}}[section]
\newtheorem{lemma}[theorem]{{Lemma}}

\newtheorem{remark}[theorem]{{Remark}}

\newtheorem{proposition}[theorem]{{Proposition}}

\def\R{\mathbb R}
\def\N{\mathbb N}

\def\Z{\mathbb Z}
\numberwithin{equation}{section}

\begin{document}

\title[Radial solutions of the Hénon equation]{Asymptotic profile and Morse index of the radial solutions of the Hénon equation}

\thanks{Wendel Leite da Silva was partially supported by CNPq and CAPES. Ederson Moreira dos Santos was partially supported by CNPq grant 307358/2015-1 and FAPESP grant 2015/17096-6.}
\author{Wendel Leite da Silva and Ederson Moreira dos Santos}
\address{
Instituto de Ci\^encias Matem\'aticas e de Computa\c{c}\~ao \\ Universidade de S\~ao Paulo,  CEP 13560-970 - S\~ao Carlos - SP - Brazil}
\email{
wendelleite94@gmail.com, ederson@icmc.usp.br}

\date{\today}

\begin{abstract}
We consider the Hénon equation
\begin{equation}\label{alphab}
-\Delta u = |x|^{\alpha}|u|^{p-1}u \ \ \textrm{in} \ \ B^N, \quad
           u = 0                     \ \ \textrm{on}\ \ \partial B^N,
           \tag{$P_{\alpha}$}
\end{equation}
where $B^N\subset \R^N$ is the open unit ball centered at the origin, $N\geq 3$, $p>1$ and $\alpha> 0$ is a parameter. We show that, after a suitable rescaling, the two-dimensional Lane-Emden equation 
\[
-\Delta w = |w|^{p-1}w\quad \text{in}\ B^2,\quad w=0\quad \text{on}\ \partial B^2,
\]
where $B^2 \subset \R^2$ is the open unit ball, is the limit problem of \eqref{alphab},  as $\alpha \to \infty$, in the framework of radial solutions. We exploit this fact to prove several qualitative results on the radial solutions of \eqref{alphab} with any fixed number of nodal sets: asymptotic estimates on the Morse indices along with their monotonicity with respect to $\alpha$; asymptotic convergence of their zeros; blow up of the local extrema and on compact sets of $B^N$.  All these results are proved for both positive and nodal solutions. 
\end{abstract}
\maketitle

\section{Introduction}
The qualitative analysis of solutions of partial differential equations is a very important field of research, and it turns out that the H\'enon equation \cite{henon} is an excellent prototype for the study of some fundamental problems in this subject. For example, the symmetry of least energy solutions and least energy nodal solutions \cite{BWW, ederson, SW, smets-su-willem} and some concentration phenomena \cite{cao-peng-yan, BW, pacella-indiana}. To investigate the symmetry of least energy solutions and least energy nodal solutions, which are known as low Morse indices solutions, a very useful argument, based on some ideas introduced in \cite{aftalion}, is that radially symmetric solutions have large Morse indices. However, the results in \cite{aftalion} are presented for autonomous problems and their proofs cannot be adapted to the nonautonomous case, which includes the H\'enon equation.  Indeed, a conjecture on the symmetry breaking of the least energy nodal solutions of the superlinear subcritical H\'enon equation  remained opened for at least 14 years, since the paper \cite{BWW}, was partially solved in \cite{ederson} for the case of $N=2$ (see also \cite{wendel}), and completely settled very recently in \cite{amadori1, amadori}. As presented ahead, as a byproduct of the main theorems in this paper, we complement some of the results in \cite{amadori2, amadori1, amadori, BW2, wendel,  weth, ederson} with new information on the asymptotic profile and Morse indices of the radial solutions of the Hénon equation, for both positive and nodal solutions.

We consider the Hénon equation
\begin{equation}\label{alpha}
\left\{
\begin{array}{l}
\begin{aligned}
-\Delta u &= |x|^{\alpha}|u|^{p-1}u &\textrm{in}&\ \ B^N, \vspace{0.3 cm}\\
           u &= 0                     &\textrm{on}&\ \ \partial B^N,
\end{aligned}
\end{array}
\right.\tag{$P_{\alpha}$}
\end{equation}
where $B^N\subset \R^N$ is the open unit ball centered at the origin, $N\geq 3$, $p>1$ and $\alpha> 0$ is a parameter. 

We recall that, for each positive integer $m$, it is shown in \cite{nagasaki} that \eqref{alpha} admits a unique radial $C^2(\,\overline{B^N}\,)$-solution, positive at zero, with exactly $m$ nodal sets, for $1<p<2^*_\alpha-1$, where $2^*_\alpha:= 2(N+\alpha)/(N-2)$. For a fixed positive integer $m$, we denote this solution by $u_{\alpha}$. Observe that the condition $p<2^{*}_{\alpha} -1$ with $\alpha>0$ is equivalent to $\alpha > \alpha_p$, where
$$
\alpha_p := \max\left\{0,\frac{p(N-2)-(N+2)}{2}\right\}.
$$

We show that, after a suitable rescaling, the two-dimensional Lane-Emden equation 
\begin{equation}\label{problema-limite}
-\Delta w = |w|^{p-1}w\quad \text{in}\ B^2,\quad w=0\quad \text{on}\ \partial B^2,
\tag{$L$}
\end{equation}
where $B^2 \subset \R^2$ is the open unit ball, is the limit problem of \eqref{alpha},  as $\alpha \to \infty$, in the framework for radial solutions.

From now on, for $m\in \N$ and $p>1$ fixed, we denote by $w$ the unique radial solution of \eqref{problema-limite} with exactly $m$ nodal sets and such that $w(0)>0$; see \cite{ni, kajikiya}. We also denote by $\lambda_1 < \ldots < \lambda_m<0$ the radial negative eigenvalues of the singular problem
 \begin{equation*}
-\Delta \psi -p|w|^{p-1} \psi = \lambda \frac{\psi}{|y|^2}\ \  \text{in}\ \ B^2\backslash\{0\},\quad  \psi=0\ \ \text{on}\ \ \partial B^2;
\end{equation*}
see \cite[Proposition 2.9]{harrabi} and \cite[Proposition 1.1]{amadori1}. To simplify notation, we identify radial functions defined on $\overline{B^N}$ with their representative defined on the interval $[0,1]$.

\begin{theorem}\label{main-theorem}
For $m\in \N$ and $p>1$ fixed, let $u_\alpha$ and $w$ be the radial solutions of \eqref{alpha} and \eqref{problema-limite}, respectively, with $m$ nodal sets such that $u_\alpha(0)>0$ and $w(0)>0$. Then 
\begin{equation}\label{v_alpha}
v_\alpha(t):= \left(\frac{2}{\alpha+2}\right)^\frac{2}{p-1} u_\alpha \left(t^\frac{2}{\alpha+2}\right),\quad t\in [0,1]
\end{equation}
converges to $w$ in $C^1([0,1]) \cap C^2([\varepsilon, 1])$ as $\alpha\to \infty$, for all $\varepsilon \in (0,1)$.  
\end{theorem}

The change of variables \eqref{v_alpha} was used in \cite{cowan}; see also \cite{gladiali-grossi-neves, wendel, ederson}. There the authors show that, in the particular case of $N=2$, for all $\alpha>0$, the function $v_\alpha$ is precisely the radial solution of \eqref{problema-limite} with $m$ nodal sets. In other words, the sequence of functions $(v_\alpha)_{\alpha>0}$ is constant if $N=2$. Although not constant for $N \geq 3$, Theorem \ref{main-theorem} shows that the sequence \eqref{v_alpha} has the same limit for every $N\geq2$. 

\begin{remark}
To prove Theorem \ref{main-theorem} we show that the absolute values of all the local extrema of $u_{\alpha}$ explode at the same rate $\alpha^{2/(p-1)}$ as $\alpha \to \infty$; see Theorem \ref{zeros-blowup} (iii) ahead, and \cite[Theorem 3.1-E]{BW} for the positive radial solution. For instance, this rate also holds for the $L^{\infty}$-norm of the least energy solutions of \eqref{alpha}; see \cite[Lemma 4.3]{cao-peng-yan} and \cite[Theorem 4.1-E]{BW}.
\end{remark}

Based on Theorem \ref{main-theorem}, we obtain some information on the Morse index of the solution $u_{\alpha}$, as $\alpha \to \infty$. Here the symbols $\lceil . \rceil$ and $\lfloor . \rfloor$ represent, respectively, the ceiling and floor functions whose definitions are given as $\lceil \beta \rceil := \min\{k\in \Z : k\geq \beta\}$, $\lfloor \beta \rfloor := \max\{k\in \Z : k\leq \beta\}$, for $\beta\in \R$.

  \begin{theorem}\label{lower-bounds-N>2}
For $m\in \N$ and $p>1$ fixed, let $m(u_\alpha)$ be the Morse index of a radial solution $u_\alpha$ of \eqref{alpha} with $m$ nodal sets. Then there exists $\alpha^*=\alpha^*(p,N,m) >\alpha_p$, that does not depend on $\alpha$, such that
\begin{equation}\label{lower-bound:Morse-index}
m(u_\alpha) \geq m + m\sum_{j=1}^{J_\alpha} N_j\quad \forall \alpha\geq \alpha^*,
\end{equation}
where 
$$J_\alpha := \left\lceil\frac{\sqrt{(N-2)^2-(\lambda_m/2)\alpha^2}}{2}\right\rceil \quad \text{and}\quad  N_j := \frac{(N+2j-2)(N+j-3)!}{(N-2)!j!}. 
$$
Moreover, if $m\geq 2$, then for each $\theta >1$, there exists $\alpha^\star= \alpha^\star(\theta,p,N,m) > \alpha_p$ independent of $\alpha$ such that
\begin{equation}\label{lower-bound:Morse-index2}
m(u_\alpha) \geq m + (m-1)\sum_{j=1}^{K_\alpha(\theta)} N_j\quad \forall \alpha\geq \alpha^\star,
\end{equation}
where 
$$K_\alpha(\theta) := \left\lceil\frac{\sqrt{(N-2)^2-(\lambda_{m-1}/\theta)\alpha^2}}{2}\right\rceil.  
$$
\end{theorem}

\begin{remark}\label{rmk1}
Theorem \ref{lower-bounds-N>2} deserves some comments. Comparing to \cite[Theorem 1.1]{amadori}, our contribution here is twofold. Firstly, our result is also valid for positive solutions ($m=1$), while \cite{amadori} gives no new information for positive solutions. Secondly, for nodal solutions, we improve the lower bounds for the Morse indices for large values of $\alpha$. Indeed, it is shown in \cite[Theorem 1.1]{amadori} that
$$
m(u_\alpha) \geq m + (m-1)\sum_{j=1}^{1+\lfloor \frac{\alpha}{2} \rfloor} N_j\quad \forall \alpha> \alpha_p. 
$$
Moreover, from \cite[Eq. (6.11)]{pacella-ianni} (see also \cite[Proposition 3.3]{amadori}), it is known that $-\lambda_{m-1}>1$ and hence for fixed $\theta_0>1$ close enough to $1$, one has that $-\lambda_{m-1}/\theta_0 > 1$ and thus $K_{\alpha}(\theta_0) > 1 + \left\lfloor\frac{\alpha}{2}\right\rfloor$ for all $\alpha$ large enough. Theorem \ref{lower-bounds-N>2} also presents an improvement on the results \cite[Theorem 1.1 i)]{weth}, once we have an explicit lower bound. Finally, for nodal solutions and large values of $\alpha$, it is clear that the lower bound in \eqref{lower-bound:Morse-index2} is much better than that in \eqref{lower-bound:Morse-index}, where the latter is also valid for positive solutions ($m=1$).
\end{remark}

In \cite{wendel} we proved that, in dimension $N=2$, the Morse index of the radial solutions of \eqref{alpha} with the same number of nodal sets is monotone nondecreasing with respect to $\alpha$. The case of $N=2$ is special, we took advantage of a change of variables that relates  Hénon equations with different weights. Although such transformation is not available for dimensions higher than two, here we prove a similar monotonicity result for $N\geq3$ and large values of $\alpha$.

 \begin{theorem}\label{monotonicity-N>2}
Let $p>1$ and $m \in \mathbb{N}$.  Let $u_\alpha$ and $u_\beta$ be radial solutions of \eqref{alpha} and $(P_\beta)$, respectively, with $m$ nodal sets. Then exists $\alpha_* >\alpha_p$ such that
$$
m(u_\alpha) \leq m(u_\beta),\quad \forall \alpha, \beta \in [\alpha_*,\infty),\ \alpha< \beta. 
$$
\end{theorem}

Next, we present a result on the asymptotic concentration of zeros and blow up of radial solutions of \eqref{alpha}, as $\alpha \to \infty$. If $m= 1$, then $x=0$ is the unique local extremum of $u_{\alpha}$. If $m\geq 2$, denote by $r_{1,\alpha}<\cdots<r_{m,\alpha}=1$ the zeros of $u_\alpha$ in $(0,1]$ and set
 $$
 \mathcal{M}_{1,\alpha} := \max\{u_\alpha(r): 0\leq r \leq r_{1,\alpha}\},\quad \mathcal{M}_{i+1,\alpha} := \max\{|u_\alpha(r)|: r_{i,\alpha}\leq r \leq r_{i+1,\alpha}\}
 $$
 for $i=1,\cdots,m-1$.

\begin{theorem}\label{zeros-blowup} For $m\in \N$ and $p>1$ fixed, let $u_\alpha$ and $w$ be the radial solutions of \eqref{alpha} and \eqref{problema-limite}, respectively, with $m$ nodal sets such that $u_\alpha(0)>0$ and $w(0)>0$. 
\begin{enumerate}[(i)]
\item For each $i=1,\cdots,m$,
$$
\alpha(1-r_{i,\alpha}) \to -2\log(t_i)\quad \text{as}\ \alpha\to\infty,
$$
where $r_{1, \alpha} < \ldots< r_{m-1, \alpha} < r_{m, \alpha}=1$ and $t_1 < \ldots< t_{m-1} < t_m=1$ are the zeros of $u_{\alpha}$ and $w$, respectively.  
\item  Let $0<R_0<1$. Then
$$
\left(\frac{2}{\alpha + 2}\right)^\frac{2}{p-1}u_{\alpha}(x) \to w(0)
$$
uniformly in $B(0,R_0) = \{ x \in \mathbb{R}^N; |x|< R_0\}$ as $\alpha \to\infty$.  

\item There exist $\alpha_0 > \alpha_p$ and constants $C_1,C_2>0$ such that
$$
C_1\left(\frac{\alpha+2}{2}\right)^\frac{2}{p-1} \leq \mathcal{M}_{i,\alpha}\leq C_2 \left(\frac{\alpha+2}{2}\right)^\frac{2}{p-1}
$$
for all $i=1,\cdots,m$ and $\alpha\geq \alpha_0$. 
\end{enumerate}
\end{theorem}

\begin{remark} We make some comments on Theorem \ref{zeros-blowup}.
\begin{enumerate}[(i)]
\item Regarding Theorem \ref{zeros-blowup}(i), in the case with $\alpha>0$ fixed and $p \nearrow 2^*_{\alpha}-1$, a different phenomenon occurs. In \cite{amadori2}, the authors studied the asymptotic behavior of radial solutions of \eqref{alpha} for $\alpha > 0$ fixed and $p \nearrow 2^*_{\alpha}-1$. In this case, if $r_{1,p} < \ldots < r_{m,p}=1$ are the zeros of $u_p$, the unique radial nodal solutions of \eqref{alpha} with $m$ nodal sets and $u_p(0)>0$,  then $r_{i,p}\to 0$ as $p\nearrow 2^*_\alpha-1$, for all $i=1,\ldots,m-1$; see \cite[Theorem 1.2]{amadori2}.

\item Let $\omega_{\alpha}$ be the least energy solution of the H\'enon equation \eqref{alpha} and $R_0 \in (0,1)$. It was proved in \cite[Eq. (32)]{BW2} that
$$
\left(\frac{2}{\alpha + 2}\right)^\frac{2}{p-1}\omega_{\alpha}(x) \to 0
$$
uniformly in $B(0,R_0)$, as $\alpha \to\infty$. According to Theorem \ref{zeros-blowup}(ii), this is in contrast with the radial, either nodal or positive, solutions of \eqref{alpha}.
\end{enumerate}
\end{remark}

For $p> 1$ and $\alpha>\alpha_p$, denote by $S_{\alpha,p}^{R}$ the best constant of the embedding $H^1_{0,rad}(B^N) \subset L^{p+1}(B^N, |x|^\alpha)$, that is
$$
S_{\alpha,p}^{R} := \inf_{0\neq u\in H^1_{0,rad}(B^N)} \displaystyle{\frac{\int_{B^N} |\nabla u|^2}{\left(\int_{B^N} |x|^\alpha |u|^{p+1}\right)^\frac{2}{p+1}}}. 
$$
Actually, the condition $\alpha > \alpha_p$ implies that the above embedding is compact, and hence $S_{\alpha,p}^{R}$ is achieved by a positive radial function $u_\alpha$, which is also, up to scaling, the positive radial solution of \eqref{alpha}. It is proved in \cite[Theorem 4.1]{smets-su-willem} that
$$
\lim_{\alpha\to\infty}\left(\frac{N}{\alpha + N}\right)^\frac{p+3}{p+1}S_{\alpha,p}^{R} = C \in (0,\infty),
$$
Here, we give a precise characterization for the constant $C$ above.  Let $S_p$ be the best constant of the embedding $H^1_0(B^2) \subset L^{p+1}(B^2)$, namely
$$
S_{p} := \inf_{0\neq v\in H^1_{0}(B^2)}\displaystyle{\frac{\int_{B^2} |\nabla v|^2}{\left(\int_{B^2} |v|^{p+1}\right)^\frac{2}{p+1}}}. 
$$

\begin{theorem}\label{theorem-constants}
For any $p>1$,
$$
\lim_{\alpha\to\infty}\left(\frac{2}{\alpha + 2}\right)^\frac{p+3}{p+1}S_{\alpha,p}^{R} = \left(\frac{\omega_{N-1}}{2\pi}\right)^\frac{p-1}{p+1} S_p.
$$
\end{theorem}

The rest of this paper is organized as follows. In Section \ref{sec:energy} we prove sharp estimates for the energy of the solution $u_{\alpha}$, which are used in Section \ref{sec:norm} to prove the asymptotic behavior of $\|u_{\alpha}\|_{\infty}$, as $\alpha\to \infty$. Section \ref{sec:local} is devoted to estimates on the local extrema of $u_{\alpha}$, in particular the proof of Theorem \ref{zeros-blowup} (iii). In Section \ref{section:asymptotic-behavior} we prove Theorems \ref{main-theorem} and \ref{theorem-constants} and conclude the proof of Theorem \ref{zeros-blowup}. Finally, Section \ref{sec:spec} is devoted to the analysis of the asymptotic distribution for the spectrum of some linearized operators and proofs of Theorems \ref{lower-bounds-N>2} and \ref{monotonicity-N>2}.

\section{Energy levels and their asymptotic estimates}\label{sec:energy}

Set
$$
\varphi_\alpha(u) := \frac{1}{2}\int_{B^N} |\nabla u|^2 dx - \frac{1}{p+1}\int_{B^N} |x|^\alpha |u|^{p+1} dx.
$$
Since the embedding $H^1_{0,rad}(B^N)\subset L^{p+1}(B^N, |x|^\alpha)$ is compact for all $1<p< 2^*_\alpha-1$, $\varphi_\alpha$ is well defined in $H^1_{0,rad}(B^N)$ and the radial solutions of \eqref{alpha} are exactly the critical points of $\varphi_\alpha$. In particular, these solutions lie on the Nehari manifold
$$
\mathcal{N}^\alpha := \left\{0\neq u \in H^1_{0,rad}(B^N): \varphi_\alpha'(u)u=0\right\} = \left\{0\neq u \in H^1_{0,rad}(B^N): \int_{B^N} |\nabla u|^2 = \int_{B^N} |x|^\alpha |u|^{p+1}\right\}. 
$$

We set
$$
H_m := \{u \in H^1_{0,rad}(B^N): \ u\ \text{has precisely}\ m\ \text{nodal sets}\}. 
$$
For $u\in H_m$, let $R_1, \cdots, R_m$ be the nodal regions of $u$. We introduce the function $u_i$ defined by
$$
u_i(x) :=
\left\{ 
\begin{array}{l}
\begin{aligned}
&u(x)\ &\text{if}&\ x\in R_i\\
 &0\   &\text{if}&\ x\in B^N\backslash R_i. 
\end{aligned}
\end{array}
\right.
$$

Since $u_\alpha$ and $-u_{\alpha}$ are the radial solutions of \eqref{alpha} with precisely $m$ nodal sets, we can define the level of radial solutions in $H_m$ as
$$
{C}_{\alpha,m} := \varphi_\alpha (u_\alpha). 
$$
We also define
$$
\overline{{C}}_{\alpha,m} := \inf \{\varphi_\alpha (u): u\in H_m\ \text{and}\ u_i \in \mathcal{N}^\alpha\ \forall i=1,\cdots, m\}
$$
and
$$
\widetilde{{C}}_{\alpha,m} := \inf_{u\in H_m} \max_{t\in \R^{m}_+} \sum_{i=1}^m \varphi_\alpha(t_i u_i),
$$
where here $\R^m_+ := \{t=(t_1,\cdots,t_m) \in \R^m: t_i>0\ \forall i=1,\cdots,m\}$.  

\begin{proposition}\label{levels}
${C}_{\alpha,m} = \overline{{C}}_{\alpha,m} = \widetilde{{C}}_{\alpha,m}$. 
\end{proposition}

To prove this proposition, we recall the procedure to produce a radial solution of \eqref{alpha} with $m$ nodal sets, provided $\alpha>\alpha_p$. The compactness of the embedding $H^1_{0,rad}(B^N)\subset L^{p+1}(B^N, |x|^\alpha)$ implies that the infimum of $\varphi_\alpha$ on $\mathcal{N}^\alpha$ is achieved at the positive radial solution of \eqref{alpha}. When $m\geq 2$, a radial solution of \eqref{alpha} with $m$ nodal sets can be obtained by the Nehari method; \cite[Chapter 4]{willem}. For that, we introduce the spaces
$$
\begin{aligned}
X_{s,t} &:= \{u \in H^1_{0,rad}(B^N): u(r)=0\ \forall r\in [0, s]\cup [t,1]\},\quad &0<s<t\leq 1,\\
X_{0,t} &:= \{u \in H^1_{0,rad}(B^N): u(r)=0\ \forall r\in [t, 1]\},\quad &0<t< 1,
\end{aligned}
$$
and the Nehari sets
$$
\mathcal{N}^\alpha_{s,t} := \left\{0\neq u \in X_{s,t}: \int_{B^N} |\nabla u|^2 dx = \int_{B^N} |x|^{\alpha} |u|^{p+1} dx\right\}
$$
for $0 \leq s < t \leq 1$, and solve the minimization problem
\begin{equation}\label{C_m}
\widehat{{C}}_{\alpha,m} := \inf \left\{\sum_{i=1}^m \min_{\mathcal{N}^\alpha(t_{i-1},t_i)} \varphi_{\alpha}: 0=t_0 < t_1<\cdots < t_m=1\right\}. 
\end{equation}
Arguing as in \cite[Proposition 4.4-(d)]{willem}, $\widehat{{C}}_{\alpha,m}$ is achieved, say at $0=r_0<r_1<\cdots<r_m=1$. Moreover, for each $i=1,\cdots,m$, there exists a nonnegative function $u_i\in X_{r_{i-1},r_i}$ such that
$$
\varphi_\alpha(u_i) = \min_{\mathcal{N}^\alpha_{r_{i-1},r_i}} \varphi_{\alpha}. 
$$
It can then be shown, like in \cite[Lemma 4.7]{willem}, that the function $u:B^N\to \R$ given by $u(x) = (-1)^{i-1} u_i(x)$ if $|x|\in [r_{i-1},r_i)$ is of class $C^ 2$ and therefore it is the radial solution of \eqref{alpha} with exactly $m$ nodal sets such that $u(0)>0$. Consequently, $u=u_\alpha$ and
\begin{equation}\label{levels2}
{C}_{\alpha,m}=\varphi_\alpha(u_\alpha) = \widehat{{C}}_{\alpha,m}. 
\end{equation}

Now we can prove Proposition \ref{levels}. 

\begin{proof}[\textbf{Proof of Proposition \ref{levels}}]
Since $u_\alpha \in H_m$ and the restrictions of $u_\alpha$ to its nodal regions lie on $\mathcal{N}^\alpha$, we have $\overline{{C}}_{\alpha,m} \leq {C}_{\alpha,m}$. By contradiction, suppose $\overline{{C}}_{\alpha,m} < {C}_{\alpha,m}$. Then, by \eqref{levels2}, there exists $u \in H_m$ with $u_i \in \mathcal{N}^\alpha$ for $i=1,\cdots,m$ such that $\varphi_\alpha(u) < \widehat{{C}}_{\alpha,m}$. Hence, if $s_1<\cdots<s_m=1$ are the zeros of $u$ in $(0,1]$ and $0=s_0$, then
$$
\widehat{{C}}_{\alpha,m} > \varphi_\alpha(u) = \sum_{i=1}^m \varphi_\alpha(u_i) \geq \sum_{i=1}^m \min_{\mathcal{N}^\alpha(s_{i-1},s_i)} \varphi_{\alpha} \geq \widehat{{C}}_{\alpha,m},
$$
which is a contradiction. 

To check the second equality, it suffices to observe that $\displaystyle\max_{t\in \R^{m}_+} \displaystyle\sum_{i=1}^m \varphi_\alpha(t_i u_i)$ is achieved at a unique point $t\in \R^m_+$ such that $t_i u_i \in \mathcal{N}^\alpha$ for all $i=1,\cdots,m$; see \cite[Eq. (4.1)]{willem}. 
\end{proof}

\begin{remark}
Let $u\in H_m$ and $t\in \R^m_+$ such that $t_i u_i \in \mathcal{N}^\alpha$ for all $i=1,\cdots,m$. Then 
$$
t_i = \left(\frac{\int_{B^N} |\nabla u_i|^2}{\int_{B^N} |x|^\alpha|u_i|^{p+1}}\right)^\frac{1}{p-1}\quad \forall i=1,\cdots,m,
$$
and hence
$$
\varphi_\alpha(t_i u_i) = \left(\frac{1}{2} - \frac{1}{p+1}\right) t_i^2 \int_{B^N} |\nabla u_i|^2 = \left(\frac{1}{2} - \frac{1}{p+1}\right) \frac{\left(\int_{B^N} |\nabla u_i|^2\right)^\frac{p+1}{p-1}}{\left(\int_{B^N} |x|^\alpha|u_i|^{p+1}\right)^\frac{2}{p-1}}. 
$$
Consequently, by Proposition \ref{levels}, we have the following characterization for ${C}_{\alpha,m}$:
\begin{equation}\label{C_alpha,m}
{C}_{\alpha,m} = \left(\frac{1}{2} - \frac{1}{p+1}\right) \inf_{u\in H_m} \sum_{i=1}^m \left[\frac{\left(\int_{B^N} |\nabla u_i|^2\right)^\frac{p+1}{p-1}}{\left(\int_{B^N} |x|^\alpha|u_i|^{p+1}\right)^\frac{2}{p-1}}\right].
\end{equation}
\end{remark}

\begin{proposition}\label{C_alpha}
There are constants $C_1,C_2>0$, that depend only on $p$, $N$ and $m$, such that 
$$
C_1 \left(\frac{\alpha + N}{N}\right)^\frac{p+3}{p-1} \leq {C}_{\alpha,m} \leq C_2 \left(\frac{\alpha + N}{N}\right)^\frac{p+3}{p-1},\quad \forall \alpha > \alpha_p.  
$$
Moreover, the function $\alpha \mapsto C_{\alpha,m}$ is strictly increasing in $(\alpha_p,\infty)$. 
\end{proposition}
\begin{proof}
Here we use some arguments from the proof of \cite[Theorem 4.1]{smets-su-willem}, where the case of positive solutions are treated. Given $u\in H_m$, define the rescaled function $v(t) = u(r)$, where $r=t^\theta$ and $\theta= \frac{N}{\alpha+N}\in (0,1]$. Of course $v$ has $m$ nodal sets, $v_i(t)=u_i(r)$ for each $i=1,\cdots,m$, and 
$$
\begin{aligned}
\int_{B^N} |x|^\alpha |u_i(x)|^{p+1} dx  = \theta\int_{B^N} |v_i|^{p+1}dx, \qquad
 \int_{B^N} |\nabla u_i(x)|^2 dx =\theta^{-1}\int_{B^N} |\nabla v_i(x)|^2 |x|^{(2-N)(1-\theta)} dx. 
 \end{aligned}
$$
Then 
$$
\frac{\left(\displaystyle\int_{B^N} |\nabla u_i|^2\right)^\frac{p+1}{p-1}}{\left(\displaystyle\int_{B^N} |x|^\alpha|u_i|^{p+1}\right)^\frac{2}{p-1}} = \frac{1}{\theta^\frac{p+3}{p-1}} \frac{\left(\displaystyle\int_{B^N} |\nabla v_i|^2 |x|^{(2-N)(1-\theta)}\right)^\frac{p+1}{p-1}}{\left(\displaystyle\int_{B^N} |v_i|^{p+1}\right)^\frac{2}{p-1}}
$$
which, by \eqref{C_alpha,m}, yields
\begin{equation}\label{c_alpha_m}
{C}_{\alpha,m} = \frac{\left(\frac{1}{2} - \frac{1}{p+1}\right) {R}_{\theta,m}}{\theta^{\frac{p+3}{p-1}}} = \left(\frac{1}{2} - \frac{1}{p+1}\right)\left(\frac{\alpha + N}{N}\right)^\frac{p+3}{p-1} {R}_{\theta,m}\,,
\end{equation}
where
$$
{R}_{\rho,m} := \inf_{v\in H_m} \sum_{i=1}^m \frac{\left(\int_{B^N} |\nabla v_i|^2 |x|^{(2-N)(1-\rho)}\right)^\frac{p+1}{p-1}}{\left(\int_{B^N} |v_i|^{p+1}\right)^\frac{2}{p-1}},\quad \rho \in [0,1]. 
$$
Since $\rho \mapsto {R}_{\rho,m}$ is non-increasing and $0< {R}_{1,m}< {R}_{0,m}<\infty$, we infer that
$$
C_1\left(\frac{\alpha + N}{N}\right)^\frac{p+3}{p-1} \leq {C}_{\alpha,m} \leq C_2\left(\frac{\alpha + N}{N}\right)^\frac{p+3}{p-1} \quad \forall \alpha>\alpha_p,
$$
where $C_1 = \left(\frac{1}{2} - \frac{1}{p+1}\right) \displaystyle\lim_{\rho \to 1^-} {R}_{\rho,m}$ and $C_ 2 = \left(\frac{1}{2} - \frac{1}{p+1}\right) \displaystyle\lim_{\rho \to 0^+} {R}_{\rho,m}$. The strict monotonicity of $\alpha \mapsto {C}_{\alpha,m}$ follows immediately from \eqref{c_alpha_m}. 
\end{proof}

\section{Estimates for the $L^{\infty}$-norms}\label{sec:norm}
 
 In this section we present asymptotic estimates for $\|u_\alpha\|_\infty$. We recall that, by \cite[Lemma 5.2]{harrabi}, $u_\alpha(0) = \displaystyle\max_{x\in B^N} u_\alpha(x) = \|u_\alpha\|_{\infty}$. Moreover, by Proposition \ref{C_alpha}, there exist constants $C_1,C_2>0$, that do not depend on $\alpha$, such that
\begin{equation}\label{limitacao-grad-u}
C_1 \left(\frac{\alpha + N}{N}\right)^\frac{p+3}{p-1} \leq \int_{B^N} |\nabla u_\alpha|^2 = \int_{B^N} |x|^\alpha |u_\alpha|^{p+1} \leq C_2 \left(\frac{\alpha + N}{N}\right)^\frac{p+3}{p-1}, \quad \forall \alpha>\alpha_p. 
\end{equation}
We start with the bound from below. 

\begin{lemma}\label{limitacao-por-baixo}
There exists a constant $C>0$, that does not depend of $\alpha$, such that
$$
\|u_\alpha\|_\infty \geq C\left(\frac{\alpha + N}{N}\right)^\frac{2}{p-1},\quad \forall \alpha>\alpha_p. 
$$ 
\end{lemma}

\begin{proof}
For each $\alpha>\alpha_p$, with $\theta= \frac{N}{\alpha + N}$, consider the function 
$$
\widetilde{u}_\alpha(y) := \theta^{\frac{2}{p-1}} u_\alpha\left(\theta^{\frac{1}{2-N}}y\right),\quad y \in \Omega_\alpha := B\left(0, \theta^{\frac{1}{N-2}}\right).
$$
Since $\|\widetilde{u}_\alpha\|_{\infty} = \left(\frac{N}{\alpha + N}\right)^\frac{2}{p-1}\|u_\alpha\|_{\infty}$, it suffices to show that $\|\widetilde{u}_\alpha\|_{\infty} \geq C$.  Performing the change of variables $x=\theta^{\frac{1}{2-N}}y$, we infer that 
$$
\begin{aligned}
\int_{\Omega_\alpha} |\nabla \widetilde{u}_\alpha(y)|^2 dy =\left(\frac{N}{\alpha + N}\right)^{\frac{p+3}{p-1}} \int_{B^N} |\nabla u_\alpha(x)|^2 dx,
\end{aligned}
$$
whence, by \eqref{limitacao-grad-u},
\begin{equation}\label{DES1}
C_1 \leq \int_{\Omega_\alpha} |\nabla \widetilde{u}_\alpha|^2 dy \leq C_2, \quad \forall \alpha>\alpha_p. 
\end{equation}
In addition, since $u_\alpha$ solves \eqref{alpha}, $\widetilde{u}_\alpha$ satisfies
\begin{equation}\label{eq:nova}
-\Delta \widetilde{u}_\alpha(y) = \theta^{\frac{-2(N-1)}{N-2}} \left|\theta^{\frac{1}{2-N}}y\right|^\alpha |{\widetilde{u}_\alpha(y)}|^{p-1}\widetilde{u}_\alpha(y),\quad  y\in \Omega_\alpha,\quad \widetilde{u}_\alpha=0\ \text{on}\ \partial \Omega_\alpha. 
\end{equation}
Consequently, by Hölder's inequality, \eqref{DES1} and \eqref{eq:nova},
\begin{equation}
\begin{aligned}\label{DES2}
0< C_1 &\leq \int_{\Omega_\alpha} |\nabla \widetilde{u}_\alpha|^2 dy = \int_{\Omega_\alpha} \theta^{\frac{-2(N-1)}{N-2}} \left|\theta^{\frac{1}{2-N}}y\right|^\alpha |{\widetilde{u}_\alpha}|^{p+1} dy\\
 &\leq \|{\widetilde{u}_\alpha}\|_{\infty}^{p-1} \int_{\Omega_\alpha} \theta^{\frac{-2(N-1)}{N-2}} \left|\theta^{\frac{1}{2-N}}y\right|^\alpha |{\widetilde{u}_\alpha}|^{2} dy\\
& \leq \|{\widetilde{u}_\alpha}\|_{\infty}^{p-1} \left(\int_{\Omega_\alpha} \!\!\!\!\! \theta^{\frac{-2(N-1)}{N-2}} \left|\theta^{\frac{1}{2-N}}y\right|^\alpha dy\right)^{\frac{p-1}{p+1}} \left(\int_{\Omega_\alpha} \!\!\!\!\! \theta^{\frac{-2(N-1)}{N-2}} \left|\theta^{\frac{1}{2-N}}y\right|^\alpha |{\widetilde{u}_\alpha}|^{p+1} dy\right)^\frac{2}{p+1}\\
&\leq (C_2)^{\frac{2}{p+1}} \|{\widetilde{u}_\alpha}\|_{\infty}^{p-1} \left(\int_{\Omega_\alpha} \!\!\!\!\! \theta^{\frac{-2(N-1)}{N-2}} \left|\theta^{\frac{1}{2-N}}y\right|^\alpha dy\right)^{\frac{p-1}{p+1}} .
\end{aligned}
\end{equation}
Note that, since $\theta= \frac{N}{\alpha + N}$,
\begin{equation}\label{IGUAL3}
\begin{aligned}
\int_{\Omega_\alpha} \theta^{\frac{-2(N-1)}{N-2}} \left|\theta^{\frac{1}{2-N}}y\right|^\alpha dy = \omega_{N-1}\theta^{\frac{2(N-1)+\alpha}{2-N}} \int_0^{\theta^{\frac{1}{N-2}}} t^{\alpha+N-1} dt = \frac{\omega_{N-1}}{N}. 
\end{aligned}
\end{equation}
Therefore, from \eqref{DES2} and \eqref{IGUAL3}, there exists $C>0$, independent of $\alpha$, such that
$$
\|{\widetilde{u}_\alpha}\|_{\infty} \geq C,\quad \forall \alpha> \alpha_p,
$$
which proves the lemma. 
\end{proof}

To prove the reverse estimate, for each $\alpha>\alpha_p$, consider 
\begin{equation}\label{v-alpha}
v_\alpha(t) := \left(\frac{2}{\alpha+2}\right)^{\frac{2}{p-1}}u_\alpha\left(t^{\frac{2}{\alpha+2}}\right),\quad t\in [0,1]. 
\end{equation}
Then by \cite[Proposition 4.2]{cowan}, $v_\alpha$ satisfies
\begin{equation}\label{EDO-v}
\left\{
\begin{array}{l}
-(t^{M_\alpha-1}v_\alpha')' = t^{M_\alpha -1}|v_\alpha|^{p-1}v_\alpha \quad \textrm{in}\ \ (0,1), \vspace{0.3 cm}\\
          v_\alpha(1)=v_\alpha'(0)=0,
\end{array}
\right.
\end{equation}
where $M_\alpha := \frac{2(\alpha+N)}{\alpha+2} \in (2, N)$. Then, integrating from $0$ to $t$, 
 \begin{equation}\label{v'}
 v_\alpha'(t) = -\frac{1}{t^{M_\alpha -1}}\int_0^t s^{M_\alpha -1}|v_\alpha(s)|^{p-1}v_\alpha (s) ds\quad \forall t\in (0,1]. 
 \end{equation}
 Moreover, performing  the change of variables $t\leftrightarrow r$, where $r=t^{\frac{2}{\alpha+2}}$, we obtain
 \begin{equation}\label{relacao-grad-uv}
 \begin{aligned}
 \int_0^1 |v_\alpha'(t)|^2 t^{M_\alpha-1}dt = \left(\frac{2}{\alpha+2}\right)^\frac{p+3}{p-1} \int_0^1 |u_\alpha'(r)|^2 r^{N-1}dr. 
 \end{aligned}
 \end{equation}

 \begin{lemma}\label{limitacao-por-cima}
There exist $\alpha_0 > \alpha_p$ and $C>0$, that do not depend on $\alpha$, such that 
 $$
 \|u_\alpha\|_\infty \leq C\left(\frac{\alpha+2}{2}\right)^{\frac{2}{p-1}} \ \ \forall \, \alpha\geq \alpha_0. 
 $$ 
 \end{lemma}
 
 \begin{proof}  By \eqref{v-alpha}, it is enough to show that 
 $$
 v_\alpha(0)= \|v_\alpha\|_\infty \leq C, \ \ \forall\, \alpha\geq \alpha_0.
 $$
 From \eqref{limitacao-grad-u} and \eqref{relacao-grad-uv}, there exists a constant $C_0>0$ independent of $\alpha$ such that
 \begin{equation}\label{lim-grad-v}
 \int_0^1 |v_\alpha'(t)|^2 t^{M_\alpha-1}dt \leq C_0\quad \forall \alpha> \alpha_p.  
 \end{equation}
 
Fix $M>2$, with $p<\frac{M}{M-2}$, and take $\alpha_0> 0$ so that $M_\alpha \leq M$ for $\alpha \geq \alpha_0$. By contradiction, suppose that $\|v_\alpha\|_{\infty}$ is not bounded in $[\alpha_0,\infty)$. Then there exists a sequence $\alpha_n \geq \alpha_0$ such that $\|v_{\alpha_n}\|_{\infty}\to\infty$ as $n\to\infty$. Set
 $$
 \overline{v}_\alpha(s) = \frac{1}{\|v_\alpha\|_{\infty}}v_\alpha \left(\|v_\alpha\|_\infty^{1-p}s\right),\quad \text{for}\ s\in \left[0, \|v_\alpha\|_\infty^{p-1}\right], \quad \text{and} \ \ t = \|v_\alpha\|_\infty^{1-p}s.
 $$
Then, since $M\geq M_{\alpha_n}$, from \eqref{lim-grad-v} and $(M-2)p-M<0$, we infer that
 $$
 \begin{aligned}
 \int_0^{\|v_{\alpha_n}\|_\infty^{p-1}} |\overline{v}_{\alpha_n}'(s)|^2 s^{M-1}ds &= \|v_{\alpha_n}\|_\infty^{(M-2)p-M} \int_0^1 |v_{\alpha_n}'(t)|^2 t^{M-1}dt\\
  &\leq \|v_{\alpha_n}\|_\infty^{(M-2)p-M} \int_0^1 |v_{\alpha_n}'(t)|^2 t^{M_{\alpha_n}-1}dt \leq C_0 \|v_{\alpha_n}\|_\infty^{(M-2)p-M} \to 0. 
 \end{aligned}
 $$
Hence $
 \overline{v}_{\alpha_n} \to 0\ \text{in}\ \mathcal{D}^{1,2}_M,$ 
 where
 $$
 \mathcal{D}^{1,2}_M := \left\{w \in L^{2^*_M}_M(0,\infty) :\quad 
\begin{aligned}
&\text{$w$ has first order weak derivative and}\\ 
&\|w\|_{\mathcal{D}^{1,2}_M}^2 := \int_0^\infty |w'(t)|^2 t^{M-1}dt <\infty
\end{aligned}
\right\}
$$
and
$$
L^{2^*_M}_M(0,\infty):= \left\{v:[0,\infty)\to \R\ \text{measurable} : \int_0^{\infty} |v(s)|^{\frac{2M}{M-2}} s^{M-1}ds<\infty\right\}. 
$$
 Since, by Sobolev's inequality, the inclusion $
 \mathcal{D}^{1,2}_M \subset L^{2^*_M}_M(0,\infty)$ is continuous, up to a subsequence, $\overline{v}_{\alpha_n} \to 0$ a.e. in $[0,\infty)$.
 
 On the other hand, 
 $$
 \overline{v}_{\alpha_n}(0) = \|\overline{v}_{\alpha_n}\|_{\infty} = 1, 
 $$
and thus, the sequence of functions $(\overline{v}_{\alpha_n})$ is uniformly bounded. 
 Moreover, by \eqref{v'},
 $$
|\overline{v}_{\alpha_n}'(s)| = \|v_{\alpha_n}\|_\infty^{-p} |v_{\alpha_n}'\left(\|v_{\alpha_n}\|_\infty^{1-p}s\right)| \leq \|v_{\alpha_n}\|_\infty^{-p}\|v_{\alpha_n}\|_\infty^{p} = 1. 
 $$
 In particular, the sequence of functions $(\overline{v}_{\alpha_n})$ is equicontinuous. 
 Consequently, by Arzelà-Ascoli theorem, up to a subsequence, $\overline{v}_{\alpha_n} \to \overline{v}$ uniformly in $[0,1]$ with $\|\overline{v}\|_{\infty} = 1$, which leads us to a contradiction. 
 \end{proof}
 
 \begin{remark}
 Notice that, unlike Lemma \ref{limitacao-por-baixo}, the conclusion of the previous lemma is not true for all $\alpha > \alpha_p$. Consider, for example, $m=1$ ($u_\alpha$ is positive) and $p=\frac{N+2}{N-2}$ ($\alpha_p = 0$). In this case,  $\|u_\alpha\|_\infty \to \infty$ as $\alpha \to 0$. Indeed, suppose that $\|u_\alpha\|_\infty$ is uniformly bounded as $\alpha\to 0$.  Thus, since $u_\alpha$ solves
 $$
-\Delta u = |x|^\alpha u^\frac{N+2}{N-2},\ u>0 \ \text{in}\ B^N,\quad u =0\ \text{on}\ \partial B^N,
 $$
we have that as $\alpha \to 0$, $u_\alpha$ converges uniformly to a nonnegative solution $u$ of the limit problem
  $$
-\Delta u = u^\frac{N+2}{N-2} \ \text{in}\ B^N,\quad u =0\ \text{on}\ \partial B^N. 
 $$
This implies that $u = 0$, which is a contradiction, because $\|u_\alpha\|_\infty \geq C$ by Lemma \ref{limitacao-por-baixo}.
 \end{remark}

 \section{Estimates for the local extrema and proof of Theorem \ref{zeros-blowup} (iii)}\label{local-extrema}\label{sec:local}
If $m= 1$, then $x=0$ is the unique local extremum of $u_{\alpha}$, and Lemmas \ref{limitacao-por-baixo} and \ref{limitacao-por-cima} give a sharp estimate for $u_{\alpha}(0)$. Next, if $m\geq 2$, denote by $r_{1,\alpha}<\cdots<r_{m,\alpha}=1$ the zeros of $u_\alpha$ in $(0,1]$ and set
 $$
 \mathcal{M}_{1,\alpha} := \max\{u_\alpha(r): 0\leq r \leq r_{1,\alpha}\},\quad \mathcal{M}_{i+1,\alpha} := \max\{|u_\alpha(r)|: r_{i,\alpha}\leq r \leq r_{i+1,\alpha}\}
 $$
 for $i=1,\cdots,m-1$. Similarly, denote by $t_{1,\alpha}<\cdots<t_{m,\alpha}=1$ the zeros in $(0,1]$ of the function $v_\alpha$ defined in \eqref{v-alpha}. Then
 \begin{equation}\label{t_i,alpha}
 t_{i,\alpha} = (r_{i,\alpha})^{\frac{\alpha+2}{2}},\quad \text{for each}\ i=1,\cdots,m. 
 \end{equation}
Moreover, define
 $$
 \mathcal{N}_{1,\alpha} := \max\{v_\alpha(t): 0\leq t \leq t_{1,\alpha}\},\quad \mathcal{N}_{i+1,\alpha} := \max\{|v_\alpha(t)|: t_{i,\alpha}\leq t \leq t_{i+1,\alpha}\},
 $$
and observe that
\begin{equation}\label{eq:nm}
 \mathcal{N}_{i,\alpha} =\left(\frac{2}{\alpha+2}\right)^\frac{2}{p-1} \mathcal{M}_{i,\alpha},\quad \text{for each}\ i=1,\cdots,m. 
\end{equation}

\begin{lemma}\label{t_1,alpha}
There exists $\delta \in (0,1)$ and $\alpha_0 > \alpha_p$ such that 
$$
t_{1,\alpha} \geq \delta \quad \forall \alpha \geq \alpha_0. 
$$
In particular, by \eqref{t_i,alpha}, $r_{i,\alpha} \to 1$ as $\alpha \to \infty$, for each $i=1,\cdots,m$. 
\end{lemma}

\begin{proof}
Let $\alpha_0 > \alpha_p$ be as in Lemma \ref{limitacao-por-cima}. By \eqref{v-alpha} and Lemmas \ref{limitacao-por-baixo} and \ref{limitacao-por-cima}, there exist constants $C_1, C_2 >0$ such that
\begin{equation}\label{limitacao-v}
C_1 \leq \|v_\alpha\|_{\infty} \leq C_2,\quad \text{for all $\alpha\geq \alpha_0$}. 
\end{equation}
Moreover, by \eqref{v'},
\begin{equation}\label{limitacao-v'}
\|{v_\alpha}'\|_{\infty} \leq (C_2)^p,\quad \text{for all $\alpha \geq \alpha_0$}.
\end{equation}
Suppose, by contradiction, that there exists a sequence $\alpha_n \geq \alpha_0$ such that $t_{1,\alpha_n} \to 0$ as $n\to \infty$. By \eqref{limitacao-v} and \eqref{limitacao-v'}, we may use Arzelà-Ascoli theorem to conclude that, up to a subsequence, $v_{\alpha_n} \to v$ uniformly in $[0,1]$ as $n\to\infty$. Note that
$$
v(0) = v(0)-v_{\alpha_n}(t_{1,\alpha_n}) = [v(0)-v(t_{1,\alpha_n})] + [v(t_{1,\alpha_n}) - v_{\alpha_n}(t_{1,\alpha_n})] = o(1),
$$
that is, $v(0)=0$. On the other hand
$$
0<C_1\leq \lim_{n\to\infty}\|v_{\alpha_n}\|_{\infty} = \lim_{n\to\infty}v_{\alpha_n}(0) = v(0),
$$
which is a contradiction. 
\end{proof}

\begin{remark}
Observe that the previous lemma guarantees asymptotic concentration, as $\alpha \to \infty$, of the zeros of the radial solutions $u_\alpha$ of \eqref{alpha}, namely $r_{i,\alpha} \to 1$ for each $i=1,\cdots,m$. We improve this result in Section \ref{section:asymptotic-behavior}, proving Theorem \ref{zeros-blowup}, where we show the exact rate at which this concentration occurs.  
\end{remark}

Next, we obtain some asymptotic estimates for the local extrema $\mathcal{M}_{i,\alpha}$ of $u_\alpha$, as $\alpha \to \infty$.

\begin{proof}[\textbf{Proof of Theorem \ref{zeros-blowup}. Part (iii)}] By \eqref{eq:nm}, it is enough to prove that
$$
C_1\leq \mathcal{N}_{i,\alpha}\leq C_2, \ \ \forall \ \ i=1,\cdots,m, \ \  \text{and} \ \  \alpha\geq \alpha_0. 
$$

By \cite[Lemma 3.1]{amadori}, $\mathcal{N}_{1,\alpha}>\mathcal{N}_{2,\alpha}>\cdots>\mathcal{N}_{m,\alpha}$ for each $\alpha> \alpha_p$. Therefore, we may take $C_2$ as the constant $C$ from Lemma \ref{limitacao-por-cima}. 

Next for simplicity, we denote by $t_\alpha = t_{m-1,\alpha}$ the largest zero of $v_\alpha$ in (0,1) so that $\mathcal{N}_{m,\alpha}=\displaystyle\max_{t_\alpha<t<1}|v_\alpha(t)|$. Multiplying \eqref{EDO-v} by $v_\alpha$ and integrating by parts from $t_\alpha$ to $1$, we get
$$
\begin{aligned}
\int_{t_\alpha}^1|v_\alpha(t)|^{p+1} t^{M_\alpha-1}dt  = \int_{t_\alpha}^1|v_\alpha'(t)|^2 t^{M_\alpha-1}dt,
\end{aligned}
$$
since $v_\alpha(t_\alpha)=v_\alpha(1)=0$. Moreover, for every $t\in(t_\alpha,1)$,
$$
\begin{aligned}
|v_\alpha(t)| &= \left|\int_{t_\alpha}^t v_\alpha'(s)ds \right| \leq \int_{t_\alpha}^1 \left(s^\frac{1-M_\alpha}{2}\right) \left(|v_\alpha'(s)| s^\frac{M_\alpha-1}{2}\right) ds \\
 &\leq \left(\int_{t_\alpha}^1  s^{1-M_\alpha}ds \right)^{1/2} \left(\int_{t_\alpha}^1 |v_\alpha'(s)|^2 s^{M_\alpha-1}ds \right)^{1/2}. 
\end{aligned}
$$
By Lemma \eqref{t_1,alpha}, there is $\delta>0$ such that $t_\alpha \geq \delta$ and hence
$$
\left(\int_{t_\alpha}^1  s^{1-M_\alpha}ds\right)^{1/2} \leq \left(\int_\delta^1 s^{1-N} ds\right)^{1/2} = \left(\frac{\delta^{2-N}-1}{N-2}\right)^{1/2} =: C. 
$$
Consequently,
$$
|v_\alpha(t)| \leq C \left(\int_{t_\alpha}^1 |v_\alpha'(s)|^2 s^{M_\alpha-1}ds \right)^{1/2} = C \left(\int_{t_\alpha}^1|v_\alpha(s)|^{p+1} s^{M_\alpha-1}ds\right)^{1/2},\quad \forall t\in (t_\alpha,1),
$$
which implies, by definition of $\mathcal{N}_{m,\alpha}$, that
$$
\mathcal{N}_{m,\alpha} \leq C \left(\int_{t_\alpha}^1|v_\alpha(s)|^{p+1} s^{M_\alpha-1}ds\right)^{1/2} \leq C  (\mathcal{N}_{m,\alpha})^{(p+1)/2},
$$
that is,
$$
\mathcal{N}_{m,\alpha} \geq \left(\frac{1}{C}\right)^{2/(p-1)}>0. 
$$
Taking $C_1=C^{2/(1-p)}$, we obtain $\mathcal{N}_{1,\alpha}>\mathcal{N}_{2,\alpha}>\cdots>\mathcal{N}_{m,\alpha} \geq C_1$, which concludes the proof. 
\end{proof}

\section{Proofs of Theorems \ref{main-theorem}, \ref{zeros-blowup} and \ref{theorem-constants}}\label{section:asymptotic-behavior}

\begin{proof}[\textbf{Proof of Theorem \ref{main-theorem}}]
From Lemma \ref{limitacao-por-cima}, $\|v_\alpha\|_{\infty}\leq C$ for all $\alpha> \alpha_0$. By \eqref{v'}, this implies that $|v'_\alpha(t)| \leq C^p|t|$ for all $t\in(0,1]$, and hence $\|v'_\alpha\|_{\infty}\leq C^p$. Moreover, since $v_\alpha$ satisfies \eqref{EDO-v}, for all $t \in(0,1]$ and $\alpha>\alpha_0$,
$$
|v_\alpha''(t)| = \left|(M_\alpha-1)\frac{v_\alpha'(t)}{t} + |v_\alpha(t)|^{p-1}v_\alpha(t)\right| \leq (N-1)C^p + C^p=N C^p,
$$
that is, $\|v_\alpha''\|_{\infty}\leq N C^p$. 

Let $(\alpha_n)$ be any sequence such that $\alpha_n \to \infty$ as $n\to \infty$. Then, by Arzelà-Ascoli theorem, up to a subsequence, $v_{\alpha_n} \to v$ and $v_{\alpha_n}' \to z$ uniformly in $[0,1]$. Then $v$ is differentiable and $v'=z$. We claim that $v\in H^1_{0,rad}(B^2)$ is a weak radial solution of \eqref{problema-limite}. Indeed, let $\psi\in C^1([0,1])$ be such that $\psi(1)=0$. Then, multiplying \eqref{EDO-v} by $\psi$ and integrating by parts, 
$$
\int_0^1v_{\alpha_n}'(t) \psi'(t) t^{M_{\alpha_n}-1}dt = \int_0^1|v_{\alpha_n}(t)|^{p-1} v_{\alpha_n}(t) \psi(t) t^{M_{\alpha_n}-1}dt,\quad \forall n\in \N. 
$$
Hence, by uniform convergence,
$$
\int_0^1v'(t) \psi'(t)\,  t\, dt = \int_0^1|v(t)|^{p-1} v(t) \psi(t) \,t\, dt,
$$
since $M_{\alpha_n} \to 2$. This implies that $v$ is a weak radial solution of \eqref{problema-limite} and, by standard elliptic regularity, $v \in C^2([0,1])$ and satisfies 
$$
v''+ \frac{v'}{t} + |v|^{p-1}v=0\quad \text{in}\ (0,1),\quad v'(0)=v(1)=0. 
$$
Given $\varepsilon \in (0,1)$, for all $t\in [\varepsilon, 1]$, we may write
$$
\begin{aligned}
|v_\alpha'' (t)- v''(t)| &\leq \frac{|v'(t)-(M_\alpha - 1)v_\alpha'(t)|}{t} + \left||v(t)|^{p-1}v(t)- |v_\alpha(t)|^{p-1}v_\alpha(t)\right|\\
& \leq \frac{|v'(t)-(M_\alpha - 1)v_\alpha'(t)|}{\varepsilon} + \left||v(t)|^{p-1}v(t)- |v_\alpha(t)|^{p-1}v_\alpha(t)\right|,
\end{aligned}
$$
which shows that $v''_{\alpha_n} \to v''$ in $C([\varepsilon,1])$. Putting all these convergences together, we infer that $v_{\alpha_n} \to v$ in $C^1([0,1]) \cap C^2([\varepsilon,1])$.

To conclude the proof, we need to check that $v=w$. Since $v$ and $w$ are radial solutions of \eqref{problema-limite} and $w$ has precisely $m$ nodal sets, it is enough to show that $v$ also has $m$ nodal sets. Up to a subsequence, we can assume that $t_{i,\alpha_n}\to t_{i} \in [0,1]$ for each $i=1,\cdots,m$. Of course $0 < \delta \leq t_{1} \leq t_{2} \leq \cdots \leq t_{m}=1$, where $\delta$ is given by Lemma \ref{t_1,alpha}. We claim that $t_{i} < t_{i+1}$ for all $i=1,\cdots,m-1$. Indeed, by contradiction, suppose that $t_{i} = t_{i+1}$ for some $i$. Let $s_{\alpha_n}\in (t_{i,\alpha_n}, t_{i+1,\alpha_n})$ so that $|v_{\alpha_n}(s_{\alpha_n})| = \mathcal{N}_{i+1,\alpha_n}$. Then, by  Theorem \ref{zeros-blowup} (iii),
$$
\left|\frac{v_{\alpha_n}(s_{\alpha_n}) - v_{\alpha_n}(t_{i,\alpha_n})}{s_{\alpha_n} - t_{i,\alpha_n}}\right| = \frac{|v_{\alpha_n}(s_{\alpha_n})|}{s_{\alpha_n} - t_{i,\alpha_n}} \geq \frac{C_1}{s_{\alpha_n} - t_{i,\alpha_n}} \to \infty,\quad n\to \infty,
$$
which contradicts $\|v'_\alpha\|_{\infty}\leq C^p$. Finally, since $|v_{\alpha_n}|>0$ in $(0, t_{1, \alpha})$ and $(t_{i,\alpha_n},t_{i+1,\alpha_n})$, for all $i =1, \ldots, m-1$, we infer that $|v| \geq 0$ in $(0, t_1)$ and $(t_{i},t_{i+1})$ for all $i =1, \ldots, m-1$. Then, by the strong maximum principle and  Theorem \ref{zeros-blowup} (iii), $|v|>0$ in each of the intervals $(0, t_1)$ and $(t_{i},t_{i+1})$ for all $i =1, \ldots, m-1$, and hence $v$ has $m$ nodal sets. 
\end{proof}

\begin{remark}\label{v_M-convergence}
Notice that, for all ${\alpha}^* > \alpha_p$, using the same arguments, it is possible to show that the functions $v_\alpha$ converge in $C^1([0,1]) \cap C^2([\varepsilon,1])$ to the function $v_{{\alpha}^*}$, as $\alpha \to {\alpha}^*$. This can be reformulated as follows. First, take into account that $v_\alpha$ is the unique solution with $m$ nodal sets of the equation
\begin{equation}\label{Q_M}
\left\{
\begin{array}{l}
-v'' - \frac{M-1}{t}v' = |v|^{p-1} v \quad \textrm{in}\ \ (0,1), \vspace{0.3 cm}\\
          v(0)> 0,\ v'(0)=v(1) = 0,
\end{array}
\right.\tag{$Q_M$}
\end{equation}
with $M=M_\alpha = \frac{2(\alpha+N)}{\alpha+2} \in (2,N)$ (see \cite[Proposition 4.2]{cowan}). Then for all $M_0 > 2$ such that $p+1< 2^*_{M_0} := \frac{2M_0}{M_0-2}$, that is, $2< M_0 < \frac{2(p+1)}{p-1}$, the solution $v_M$ of \eqref{Q_M} with $m$ nodal sets converges in $C^1([0,1]) \cap C^2([\varepsilon,1])$ to the solution with $m$ nodal sets of $(P_{M_0})$, as $M \to M_0$. 

\end{remark}

\begin{proof}[\textbf{Proof of Theorem \ref{zeros-blowup}}]
\textbf{Part $\textbf{(i)}$.} By Theorem \ref{main-theorem}, $t_{i,\alpha}\to t_i$, for each $i=1,\cdots,m-1$, where $t_{i,\alpha}$ are the zeros of $v_\alpha$ defined in \eqref{v_alpha}. Since $t_{i,\alpha}=r_{i,\alpha}^\frac{\alpha+2}{2}$,
$$
\frac{\alpha+2}{2}(1-r_{i,\alpha}) = \frac{\alpha+2}{2}\left(1-t_{i,\alpha}^\frac{2}{2+\alpha}\right) = \frac{\alpha+2}{2}\left(1-t_{i}^\frac{2}{2+\alpha} + t_{i}^\frac{2}{2+\alpha} -t_{i,\alpha}^\frac{2}{2+\alpha}\right). 
$$
Note that, by the mean value theorem, there is $c_{i,\alpha}$ between $t_i$ and $t_{i,\alpha}$ such that
$$
\frac{\alpha+2}{2}\left(t_{i}^\frac{2}{\alpha+2} -t_{i,\alpha}^\frac{2}{\alpha+2}\right) = \frac{\alpha+2}{2} \left(\frac{2}{\alpha+2} c_{i,\alpha}^\frac{-\alpha}{\alpha+2}(t_i - t_{i,\alpha}) \right) = c_{i,\alpha}^\frac{-\alpha}{\alpha+2}(t_i - t_{i,\alpha}) \to 0,
$$
because $t_{i,\alpha}\to t_i$ and $c_{i,\alpha}^\frac{-\alpha}{\alpha+2} \to 1/t_i$.  Therefore, as $\alpha \to \infty$,
$$
\frac{\alpha+2}{2}(1-r_{i,\alpha}) = \frac{\alpha+2}{2}\left(1-t_{i}^\frac{2}{\alpha+2}\right) + o(1) \to -\log (t_i),
$$
and hence
\[
\alpha(1-r_{i,\alpha}) \to -2\log(t_i),\quad \text{as}\ \alpha\to\infty. 
\]

\medskip \noindent \textbf{Part $\textbf{(ii)}$. } First we prove that for each $x \in B^N$ fixed, \begin{equation}\label{limit}
\lim_{\alpha\to \infty} \left(\frac{2}{\alpha + 2}\right)^\frac{2}{p-1}u_{\alpha}(x) = w(0) = \|w\|_\infty.
\end{equation}
Indeed, fixed $x\in B^N$, set $r=|x|\in [0,1)$. Since $(v_\alpha)$ converges uniformly to $w$, 
$$
\left(\frac{2}{\alpha + 2}\right)^\frac{2}{p-1}u_{\alpha}(r) = v_\alpha(r^{\frac{\alpha+2}{2}}) = w(r^{\frac{\alpha+2}{2}}) + \left[v_\alpha(r^{\frac{\alpha+2}{2}}) - w(r^{\frac{\alpha+2}{2}})\right]
 = w(0) + o(1),
$$
as $\alpha \to \infty$, which proves \eqref{limit}.

Next, we prove the $L^{\infty}_{loc}$ blow-up of $u_{\alpha}$. From the equation
$$
-(r^{N-1}u_\alpha')' = r^{N-1} |u_\alpha|^{p-1}u_\alpha,\quad u_\alpha(0)>0,\quad u_\alpha'(0)=u_\alpha(1)=0,
$$
we infer that
$$
u_\alpha'(r) = -\frac{1}{r^{N-1}}\int_0^r s^{N-1} |u_\alpha(s)|^{p-1}u_\alpha(s) ds,\quad \forall r\in (0,1). 
$$
By (i), $r_{1,\alpha}\to 1$ as $\alpha\to\infty$. Hence, given $0<R_0<1$, there is $\alpha_1>\alpha_p$ such that $r_{1,\alpha} > R_0$ for all $\alpha> \alpha_1$. Consequently, $u_\alpha(r)>0$ for all $r\in [0,R_0]$ and $\alpha>\alpha_1$. Thus $u_\alpha'(r)<0$ for $r\in [0,R_0]$ and $\alpha>\alpha_1$. Set $W_\alpha (r):=\left(\frac{2}{\alpha + 2}\right)^\frac{2}{p-1}u_{\alpha}(r)$. Then, given $\varepsilon>0$, by \eqref{limit}, there is $\alpha_2>\alpha_1$ such that
$$
|W_\alpha (0) - W_\alpha(R_0)| < \varepsilon/2,\quad |W_\alpha (R_0) - w(0)| < \varepsilon/2 \quad \forall \alpha>\alpha_2. 
$$
Theorefore, the monotonicity of $W_\alpha$ in $[0,R_0]$ for $\alpha>\alpha_2$ implies that for any $r\in [0,R_0]$
$$
\begin{aligned}
|W_\alpha (r) - w(0)| &\leq |W_\alpha (r) - W_\alpha(R_0)| + |W_\alpha (R_0) - w(0)| = W_\alpha (r) - W_\alpha(R_0) + |W_\alpha (R_0) - w(0)|\\
 &\leq W_\alpha (0) - W_\alpha(R_0) + |W_\alpha (R_0) - w(0)| < \varepsilon, \quad \forall \alpha>\alpha_2. 
\end{aligned}
$$
From this we get the uniform convergence, since $\alpha_2$ does not depend on $r$. 
\end{proof}

\begin{proof}[\textbf{Proof of Theorem \ref{theorem-constants}}]
By compact embedding, $S_p$ is achieved at a positive solution $w$ of \eqref{problema-limite}, which,  by \cite{gidas-ni-nirenberg}, is radially symmetric. Moreover, by \eqref{relacao-grad-uv} and Theorem \ref{main-theorem}, as $\alpha \to \infty$,
$$
\begin{aligned}
\left(\frac{2}{\alpha+2}\right)^\frac{p+3}{p-1} \int_{B^N} |\nabla u_\alpha|^2 &= \omega_{N-1} \int_0^1 |v_\alpha'(t)|^2 t^{M_\alpha-1}dt\\
 &\to \omega_{N-1} \int_0^1 |w'(t)|^2\, t\,dt = \frac{\omega_{N-1}}{2\pi} \int_{B^2} |\nabla w|^2. 
\end{aligned}
$$
Then, since
$$
\int_{B^N} |\nabla u_\alpha|^2 = (S_{\alpha,p}^{R})^\frac{p+1}{p-1} \quad \text{and}\quad \int_{B^2} |\nabla w|^2 = (S_{p})^\frac{p+1}{p-1}, 
$$
 we infer that
\[
 \left(\frac{2}{\alpha+2}\right)^\frac{p+3}{p+1} S_{\alpha,p}^{R} \to \left(\frac{\omega_{N-1}}{2\pi}\right)^\frac{p-1}{p+1} S_{p},\quad \text{as} \ \ \alpha \to \infty. \qedhere
\]
\end{proof}

\section{Asymptotic distribution for the spectrum of the linearized operators and proofs of Theorems \ref{lower-bounds-N>2} and \ref{monotonicity-N>2}} \label{sec:spec} Here we analyze the spectrum of some linear operators. More precisely, we derive \linebreak asymptotic expansions of the negative eigenvalues of the linearized operators associated to the Hénon equation \eqref{alpha} at the radial solution $u_\alpha$, with $m$ nodal sets, as $\alpha\to \infty$. For this, consider the linearized operators $L^\alpha : H^2(B^N)\cap H^1_0(B^N) \to L^2(B^N)$ given by
$$
\varphi \mapsto L^\alpha \varphi := -\Delta \varphi -p|x|^\alpha |u_\alpha|^{p-1}\varphi,\quad \alpha > \alpha_p.
$$
In order to obtain a more precise description of the distribution of the eigenvalues of $L^\alpha$, as $\alpha \to\infty$, we consider the singular eigenvalue problem

\begin{equation}\label{sing-eigenvalue}
L^\alpha \varphi = {\widehat{\Lambda}} \frac{\varphi}{{|x|^2}}.
\end{equation}
Since $N\geq 3$, we recall that, due to the Hardy inequality, \eqref{sing-eigenvalue} is well defined in $H^1_0(B^N)$. Moreover, $\widehat{\Lambda}$ is an eigenvalue for \eqref{sing-eigenvalue} if there exists $0 \neq \varphi \in H^1_0(B^N)$ such that
$$
\int_{B^N} \nabla \varphi \nabla \phi -p|x|^\alpha |u_\alpha|^{p-1} \varphi \phi \,dx = \widehat{\Lambda} \int_{B^N} \frac{\varphi \phi}{|x|^2} \,dx,\quad \forall \phi \in H^1_0(B^N). 
$$
In addition, from \cite[Proposition 4.1]{amadori1}, it is known that each of these negative eigenvalues is given (and vice versa) by the following decomposition
\begin{equation}\label{decomposition2}
\widehat{\Lambda} = \widehat{\Lambda}^{rad} + j(N-2+j),
\end{equation}
where $\widehat{\Lambda}^{rad}$ is a negative radial eigenvalue of \eqref{sing-eigenvalue} and $j$ is some nonnegative integer. So, the negative eigenvalues of \eqref{sing-eigenvalue} can be given in terms of its negative radial eigenvalues. For this reason, we study the asymptotic distribution of negative radial eigenvalues for \eqref{sing-eigenvalue}. 

 Since $u_\alpha$ has $m$ nodal sets,  \eqref{sing-eigenvalue} admits precisely $m$ negative radial eigenvalues; see \cite[Theorem 1.3]{amadori}. Denote these eigenvalues by ${\widehat{\Lambda}}_{1,\alpha}< \cdots < {\widehat{\Lambda}}_{m,\alpha}$, and by $\varphi_{i,\alpha}$ their respective eigenfunctions. 
 
For each $M\geq 2$, set
$$
H^1_{0,M} := \left\{w:[0,1]\to \R\ \text{measurable}:\quad 
\begin{aligned}
&\text{$w$ has first order weak derivative, $w(1)=0$ and}\\ 
&\int_0^1 |w'(t)|^2 t^{M-1}dt +\int_0^1 w^2(t)t^{M-3}dt <\infty
\end{aligned}
\right\}, 
$$
which is a Hilbert space with inner product
$$
(z,v)\mapsto \int_0^1 z'(t)v'(t) t^{M-1}dt +\int_0^1 z(t)v(t)t^{M-3}dt. 
$$
 Observe, in particular, that we may rewrite $H^1_{0,2}$ as
\[
H^1_{0,2} := \left\{ u \in H^1_{0, rad}(B^2);  \int_{B^2} \frac{u^2}{|y|^2} < \infty\right\}.
\]

 Now let $w$ be a radial solution of \eqref{problema-limite} with $m$ nodal sets, and consider the singular eigenvalue problem associated to \eqref{problema-limite} at $w$, namely
   \begin{equation*}
-\Delta \psi -p|w|^{p-1} \psi = \lambda \frac{\psi}{|y|^2}\ \  \text{in}\ \ B^2\backslash\{0\},\quad  \psi\in H^1_{0,2},
\end{equation*}
which also has  $m$ negative radial eigenvalues, say $\lambda_1<\cdots<\lambda_m$.

\begin{lemma}\label{limit-eigenvalues}
$$
\lim_{\alpha\to\infty}\left(\frac{2}{\alpha+2}\right)^2 \widehat{\Lambda}_{i,\alpha} = \lambda_i, \quad \forall\,  i=1,\cdots,m. 
$$
\end{lemma}

\begin{proof}
As before, we perform the change of variables $t = r^\frac{\alpha+2}{2}$, $r=|x|$, and write
 $$
 v_\alpha (t) = \left(\frac{2}{\alpha+2}\right)^\frac{2}{p-1} u_\alpha(r),\quad \psi_{i,\alpha}(t) = \varphi_{i,\alpha}(r),\quad t\in(0,1). 
 $$
Since $\varphi_{i,\alpha}$ is radial and solves \eqref{sing-eigenvalue} with $\widehat{\Lambda}={\widehat{\Lambda}}_{i,\alpha}$, using \cite[Proposition 4.2]{cowan}, $\psi_{i,\alpha}$ satisfies
 \begin{equation}\label{EDO-psi}
-\psi_{i,\alpha}'' - \frac{M_\alpha-1}{t} \psi_{i,\alpha}'  -p|v_\alpha|^{p-1} \psi_{i,\alpha} = \lambda_{i,\alpha}\frac{\psi_{i,\alpha}}{t^2},\quad t\in (0,1),\quad \psi_{i,\alpha}(1)=0,
\end{equation}
with $M_\alpha = \frac{2(\alpha+N)}{\alpha+2}$ and $\lambda_{i,\alpha} = \left(\frac{2}{\alpha+2}\right)^2 \widehat{\Lambda}_{i,\alpha}$. Recall that, see \cite[Eq. (3.24)]{amadori1}, each eigenvalue $\lambda_{i,\alpha}$ is characterized by
$$
\lambda_{i,\alpha} = \min_{\overset{Z \subset H^1_{0,M_\alpha}}{\dim Z = i}}\max_{\overset{z\in Z}{z\neq 0}} \frac{\int_0^1 [|z'(t)|^2-p |v_\alpha|^{p-1}z^2(t)]t^{M_\alpha-1} dt}{\int_0^1 z^2(t) t^{M_\alpha-3} dt}.
$$
and this min-max problem is attained at the subspace $Z_i = span\{ \psi_{1, \alpha}, \ldots, \psi_{i,\alpha}\}$, at $z_i = \psi_{i,\alpha}$. By \cite[Proposition 3.8]{amadori1}, the eigenfunctions $\psi_{k,\alpha}$ lie on $H^1_{0,2}\subset H^1_{0, M_{\alpha}}$, for all $k \in\{1, \ldots, m\}$ and $\alpha >\alpha_p$.   Therefore,
\begin{equation}\label{H102-eigen}
\lambda_{i,\alpha} = \min_{\overset{Z \subset H^1_{0,2}}{\dim Z = i}}\max_{\overset{z\in Z}{z\neq 0}} \frac{\int_0^1 [|z'(t)|^2-p |v_\alpha|^{p-1}z^2(t)]t^{M_\alpha-1} dt}{\int_0^1 z^2(t) t^{M_\alpha-3} dt},\quad \forall i=1,\cdots,m. 
\end{equation}
By Theorem \ref{main-theorem}, $|v_\alpha|^{p-1}\to |w|^{p-1}$ uniformly in $[0,1]$, as $\alpha\to \infty$.  Similarly to \cite[Lemma 3.5]{weth-bifurcation}, since $M_\alpha\to 2$ as $\alpha \to \infty$, this implies that 
$$
\lambda_{i,\alpha} \to \min_{\overset{Z \subset H^1_{0,2}}{\dim Z = i}}\max_{\overset{z\in Z}{z\neq 0}} \frac{\int_0^1 [|z'(t)|^2-p |w|^{p-1}z^2(t)]\,t\, dt}{\int_0^1 z^2(t) t^{-1} dt},\quad \forall i=1,\cdots,m, \ \ \text{as} \ \ \alpha\to \infty,
$$
that is, $\lambda_{i,\alpha} \to \lambda_{i}$ for all $i$, as we wanted, since $\lambda_{i,\alpha} = \left(\frac{2}{\alpha+2}\right)^2 \widehat{\Lambda}_{i,\alpha}$. 
 \end{proof} 
 
 \begin{remark}\label{expansions}
 By the previous lemma, 
 $$
 \frac{\widehat{\Lambda}_{i,\alpha}}{\alpha^2} \to \frac{\lambda_i}{4},\quad \alpha\to\infty,\quad \forall i=1,\cdots,m,
 $$
 and therefore 
 \begin{equation}\label{asymp-distribution}
 \widehat{\Lambda}_{i,\alpha} = \frac{\lambda_{i}}{4}\alpha^2 + o(\alpha^2),\quad \alpha \to \infty,\quad i=1,\cdots,m. 
 \end{equation}
 \end{remark}

\begin{proof}[\textbf{Proof of Theorem \ref{lower-bounds-N>2}}]
We first recall that the Morse index of $u_\alpha$ is precisely the number of negative eigenvalues (with their multiplicity) of \eqref{sing-eigenvalue}. Moreover, each of these negative eigenvalues is given (and conversely) by the decomposition \eqref{decomposition2}. Thus, for each $i=1,\cdots,m$, we need to know the numbers $j\in \N$ that satisfy
\begin{equation}\label{decomposition-henon}
\widehat{\Lambda}_{i,\alpha} + j(N-2+j)<0. 
\end{equation}
Using \eqref{asymp-distribution}, the above inequality is equivalent to
$$
j< \frac{\sqrt{(N-2)^2 - (N-2)\lambda_i \alpha^2 + o(\alpha^2)}-(N-2)}{2}. 
$$
Since $0<-\lambda_m/2<-(N-2)\lambda_i$ for $N\geq 3$ and $i=1,\cdots,m$, there exists $\alpha^* >\alpha_p$ such that 
$$
J_\alpha \leq \frac{\sqrt{(N-2)^2-(\lambda_m/2)\alpha^2}}{2} +1 < \frac{\sqrt{(N-2)^2 - (N-2)\lambda_i \alpha^2 + o(\alpha^2)}-(N-2)}{2},  
$$
for all $\alpha\geq \alpha^*$ and $i=1,\cdots,m$. Consequently, whenever $j\leq J_\alpha$ with $\alpha\geq \alpha^*$, we get \eqref{decomposition-henon} for all $i=1,\cdots,m$. 

To show \eqref{lower-bound:Morse-index}, observe that the radial eigenvalues of \eqref{sing-eigenvalue} are simple and the numbers $j(N-2+j)$, with $j\in \N$, are eigenvalues of the Laplace-Beltrami operator $-\Delta_{S^{N-1}}$ whose multiplicity is $N_j$. It follows then, from \cite[Proposition 4.1]{amadori1}, that each eigenvalue $\widehat{\Lambda}_{i,\alpha} + j(N-2+j)$ has multiplicity $N_j$. Therefore we conclude that the linearized operator $L^\alpha$ has at least $m\sum_{j=1}^{J_\alpha} N_j$ negative eigenvalues (with their multiplicity) associated to nonradial eigenfunction for all $\alpha\geq \alpha^*$. Adding the number of negative radial eigenvalues, which is $m$, we obtain \eqref{lower-bound:Morse-index}. 

Similarly, we can check \eqref{lower-bound:Morse-index2}. Indeed, given $\theta > 1$, then $0<-\lambda_{m-1}/\theta<-(N-2)\lambda_i$ for $N\geq 3$ and $i=1,\cdots,m-1$. Thus, there exists $\alpha^\star=\alpha^\star(\theta) >\alpha_p$ such that
$$
K_\alpha(\theta) < \frac{\sqrt{(N-2)^2 - (N-2)\lambda_i \alpha^2 + o(\alpha^2)}-(N-2)}{2}\quad \forall \alpha\geq \alpha^\star,\ \ i=1,\cdots,m-1.   
$$
This implies that \eqref{decomposition-henon} holds for all $i=1,\cdots,m-1$, whenever $j\leq K_\alpha(\theta)$ with $\alpha\geq \alpha^\star$, and hence \eqref{lower-bound:Morse-index2} follows.
\end{proof}

To prove the monotonicity of the Morse indices with respect to $\alpha$ we need to improve the asymptotic expansions \eqref{asymp-distribution}.  It is proved in \cite{weth-bifurcation} that the functions $\alpha \mapsto \widehat{\Lambda}_i(\alpha) := \widehat{\Lambda}_{i,\alpha}$, $i= 1,\cdots,m$, belong to $C^1(\alpha_p,\infty)$ and satisfy
\begin{equation}\label{asymp-distribution-weth}
 \widehat{\Lambda}_i(\alpha) = \nu_i^* \alpha^2 + c_i^* \alpha + o({\alpha})\quad \text{and}\quad \widehat{\Lambda}_i'(\alpha) = 2\nu_i^* \alpha + c_i^* + o(1),
\quad \text{as}\ \alpha \to \infty,
\end{equation}
where $c_i^*$, $i=1,\cdots,m$, are constants and the values $\nu_i^*$ are the negative eigenvalues of a suitable one-dimensional eigenvalue problem; see \cite[Theorem 1.3]{weth-bifurcation}. In terms of the eigenvalues $\lambda_i$, $i=1,\cdots,m$, we may write the asymptotic expansions \eqref{asymp-distribution-weth} as follows. For each $M\geq 2$ with $p (M-2)< M+2$, let $v_M$ be the unique solution with $m$ nodal sets of \eqref{Q_M} such that $v_M(0)>0$. Then the eigenvalue problem
\begin{equation}\label{EDO-psi-2}
-\psi'' - \frac{M-1}{t}\psi' -p|v_M|^{p-1} \psi = \mu \frac{\psi}{t^2},\quad t\in (0,1),\quad  \psi\in H^1_{0,M},
\end{equation}
has exactly $m$ negative eigenvalues, say $\mu_1(M) < \cdots < \mu_m(M) < 0$. Observe that, putting $M=M_\alpha = \frac{2(\alpha+N)}{\alpha+2}$, the problems \eqref{EDO-psi-2} and \eqref{EDO-psi} are the same. Thus
$$
\mu_i(M_\alpha) = \lambda_{i,\alpha} = \left(\frac{2}{\alpha+2}\right)^2\widehat{\Lambda}_i(\alpha), 
$$
and therefore, the maps $M\mapsto \mu_i(M)$ are $C^1$-functions for each $i=1,\cdots,m$.

\begin{proposition}\label{expansion}
The $C^1$-functions $\alpha \mapsto \widehat{\Lambda}_i(\alpha)$, $i=1,\cdots,m$, satisfy
$$
 \widehat{\Lambda}_i(\alpha) := \widehat{\Lambda}_{i,\alpha} = \frac{\lambda_i}{4} \alpha^2 + c_i \alpha + o({\alpha})\quad \text{and}\quad \widehat{\Lambda}_i'(\alpha) = \frac{\lambda_i}{2} \alpha + c_i + o(1)
\quad \text{as}\ \alpha \to \infty,
$$
where $c_i= \frac{(N-2)\mu_i'(2)}{2} + \lambda_i$, $i=1,\cdots,m$. 
\end{proposition}

\begin{proof}
We may write 
$$
\mu_i(M) = \mu_i(2) + \mu_i'(2)(M-2) + o(M-2)\quad \text{and}\quad \mu_i'(M) = \mu_i'(2) + o(1)\quad \text{as } M\to 2^+. 
$$
Note that $M_\alpha \to 2^+$ as $\alpha \to \infty$. So, performing the transformation $M \leftrightarrow M_\alpha$, a $o(M - 2)$-function as $M \to 2^+$ corresponds to a  $o(M_\alpha - 2) = o(\frac{1}{\alpha})$-function as $\alpha \to \infty$. Hence 
$$
\mu_i(M_\alpha) = \mu_i(2) + \frac{2(N-2)\mu_i'(2)}{\alpha+2} + o\left(\frac{1}{\alpha}\right)\quad \text{and}\quad \mu_i'(M_\alpha) = \mu_i'(2) + o(1)\quad \text{as } \alpha \to \infty. 
$$
Therefore
$$
\begin{aligned}
\widehat{\Lambda}_i(\alpha) &= \left(\frac{\alpha+2}{2}\right)^2 \mu_i(M_\alpha) = \frac{\mu_i(M_\alpha)}{4}\alpha^2 + \mu_i(M_\alpha)\alpha + \mu_i(M_\alpha)\\
&= \frac{\mu_i(2)}{4}\alpha^2 + \frac{(N-2)\mu_i'(2)}{2}\frac{\alpha^2}{\alpha+2}+ \mu_i(2)\alpha + o(\alpha)\\
&= \frac{\mu_i(2)}{4}\alpha^2 + \left[\frac{(N-2)\mu_i'(2)}{2} + \mu_i(2)\right]\alpha + o(\alpha),\quad \text{as $\alpha\to \infty$},
\end{aligned}
$$
and 
$$
\begin{aligned}
\widehat{\Lambda}_i'(\alpha) &= \frac{\alpha+2}{2} \mu_i(M_\alpha) - \left(\frac{\alpha+2}{2}\right)^2 \mu_i'(M_\alpha) \frac{2(N-2)}{(\alpha+2)^2}\\
&= \frac{\alpha+2}{2} \mu_i(2) + (N-2)\mu_i'(2) - \frac{(N-2)\mu_i'(2)}{2} + o(1)\\
&= \frac{\mu_i(2)}{2} \alpha + \mu_i(2) + \frac{(N-2)\mu_i'(2)}{2} + o(1),\quad \text{as } \alpha\to \infty. 
\end{aligned}
$$
Finally, observing that $\mu_i(2) = \lambda_i$, we conclude the proof. 
\end{proof}

\begin{proof}[\textbf{Proof of Theorem \ref{monotonicity-N>2}}]
It is a consequence of \eqref{decomposition2} and Proposition \ref{expansion}. Indeed, since $\lambda_i <0$, by Proposition \ref{expansion},
$$
\widehat{\Lambda}_i'(\alpha) \to -\infty,\quad \alpha \to \infty,\quad \forall i=1,\cdots,m.
$$
This shows in particular that there exists $\alpha_* > \alpha_p$ such that $\widehat{\Lambda}_i'(\alpha)<0$ for all $\alpha\geq \alpha_*$, that is, the function $\alpha \mapsto \widehat{\Lambda}_{i,\alpha}$ is strictly decreasing in $[\alpha_*,\infty)$ for each $i=1,\cdots,m$. This implies, using \eqref{decomposition2}, that the Morse index $m(u_\alpha)$ is nondecreasing with respect to $\alpha$ in $[\alpha_*,\infty)$. 
\end{proof}

\end{document}